\documentclass[reqno]{amsart} 
\usepackage{enumerate,amsmath,amssymb, color} 
\usepackage{hyperref}

\theoremstyle{plain} 
\newtheorem{step}{Step} 
\newtheorem{thm}{Theorem}[section] 
\newtheorem{theorem}[thm]{Theorem} 
\newtheorem{cor}[thm]{Corollary} 
\newtheorem{corollary}[thm]{Corollary} 
 
\newtheorem{lemma}[thm]{Lemma} 
 
\newtheorem{proposition}[thm]{Proposition} 
 
\theoremstyle{remark} 
 
\newtheorem{remark}[thm]{Remark}

\theoremstyle{definition} 
 
\newtheorem{definition}[thm]{Definition} 

\def\al{{\alpha}}
\def\be{{\beta}}

\def\om{{\omega}}
\def\Om{{\Omega}}

\let\La\Lambda

\def\si{{\sigma}}

\def\epsilon{{\varepsilon}}
\def\ep{{\varepsilon}}

\def\phi{{\varphi}}

\DeclareMathAlphabet{\doba}{U}{msb}{m}{n}

\gdef\mR{\doba{R}}

\def\M{\mathcal{M}_M}

\newcommand{\xp}[2]{X_+^{#1}(#2)} 
\newcommand{\xm}[2]{X_-^{#1}(#2)} 
\newcommand{\xpm}[2]{X_{\pm}^{#1}(#2)}

\def\vol{{\mathop{\rm vol}}}
\def\scal{{s}}

\let\<\langle 
\let\>\rangle

\newcommand{\definedas}{\mathrel{\raise.095ex\hbox{\rm :}\mkern-5.2mu=}}

\newcounter{mnotecount}[section]

\def\supp{\mathrm{supp}}

\begin{document}

\title{Mass functions of a compact manifold}


\author{Andreas Hermann}
\address{Andreas Hermann,
  Institut f\"ur Mathematik,
  Universit\"at Potsdam,
  Karl-Liebknecht-Str.\,24-25,
  14476 Potsdam-Golm
  Germany} 
\email{hermanna@uni-potsdam.de} 

\author{Emmanuel Humbert}
\address{Emmanuel Humbert, LMPT \\
Universit\'e de Tours,
Parc de Grandmont,
   37200 Tours,
   France} 
\email{emmanuel.humbert@lmpt.univ-tours.fr}


\begin{abstract}
Let $M$ be a compact manifold of dimension $n$.  In this paper, we introduce the {\em Mass Function}  $a \geq 0 \mapsto \xp{M}{a}$ (resp. $a \geq 0  \mapsto \xm{M}{a}$) which is defined 
as the supremum (resp. infimum) of the masses of all metrics on $M$  whose Yamabe constant is larger than $a$ and which are flat on a ball of radius~$1$ and centered at  a point $p \in M$. Here, the mass of a metric flat around~$p$ is the constant term in the expansion of the Green function of the conformal Laplacian at~$p$. We show that these functions are well defined and have many properties which allow to obtain applications to the Yamabe invariant (i.e. the supremum of Yamabe constants over the set of all metrics on $M$). 
\end{abstract}

\maketitle

\tableofcontents


\section{Introduction} 

Let $(M,g)$ be a closed Riemannian manifold of dimension $n\geq 3$ and denote by
\begin{equation*}
  L_g:=\Delta_g+\frac{n-2}{4(n-1)}\scal_g: \quad C^{\infty}(M)\to C^{\infty}(M)
\end{equation*}
the conformal Laplace operator of $g$, where $\scal_g$ is the scalar curvature of $g$ and $\Delta_g$ is the Laplace-Beltrami operator with non-negative spectrum. 
Assume that the metric $g$ is flat on an open neighborhood of a point $p\in M$ and that all eigenvalues of $L_g$ are strictly positive.
Then it is well-known that there exists a unique Green function $G_g$ of $L_g$ at $p$, i.\,e.\,in the sense of distributions we have $L_gG_g=\delta_p$, the function $G_g$ is smooth and strictly positive on $M\setminus\{p\}$ and as $x\to p$ we have 
\begin{equation*}
G_g(x)=\frac{1}{(n-2)\omega_{n-1}r(x)^{n-2}}+m(g,p)+o(1)
\end{equation*}
where $\omega_{n-1}$ is the volume of the standard sphere of dimension $n-1$, the function $r$ denotes the Riemannian distance from $p$ and $m(g,p)\in\mR$ is a number called the mass of $g$ at $p$. 
This quantity is related to the so-called ADM mass of an asymptotically flat Riemannian manifold.
The study of the mass has led to many interesting results in geometric analysis and General Relativity.
An example is an application to the so-called conformal Yamabe constant of $(M,g)$ defined by
\begin{equation*}
Y(M,g):=\inf \int_{M}\scal_g\,dv^g,
\end{equation*}
where the $\inf$ is taken over the set of all Riemannian metrics on $M$ which have unit volume and are conformal to $g$.
Namely, in a famous article \cite{schoen:84}, Richard Schoen used positivity of the mass $m(g,p)$ to prove that $Y(M,g)<Y(S^n,g_{\mathrm{can}})$ if $(M,g)$ satisfies the assumptions above and is not conformally diffeomorphic to the standard sphere $(S^n,g_{\mathrm{can}})$.

In this article we consider the dependence of the mass on the Yamabe constant $Y(M,g)$. 
We define two functions $a\mapsto\xp{M}{a}$ and $a\mapsto\xm{M}{a}$ whose values are a $\sup$ and an $\inf$ of masses $m(g,p)$ respectively taken over the set of all Riemannian metrics $g$ with $Y(M,g)>a$ which are flat on a ball of radius $1$ centered at $p\in M$ (see Definition~\ref{def:mass_functions}).
We prove that for small values of $a$ the values $\xp{M}{a}$ and $\xm{M}{a}$ decrease and increase respectively under surgery of codimension at least $3$ (see Theorem~\ref{main_surgery}).
Finally, we give an application to the smooth Yamabe invariant of $M$ defined by 
\begin{equation*}
\sigma(M):=\sup Y(M,g)
\end{equation*}
where the $\sup$ is taken over the set of all Riemannian metrics on $M$.
The question of whether for a given smooth manifold $M$ one has $\sigma(M)<\sigma(S^n)$ is open in general.
We prove that if $\xp{M}{\sigma(M)}>0$ then we have $\sigma(M)<\sigma(S^n)$. 
The precise statement is given in Theorem \ref{thm:xp_sigma}. 

In the proofs of these theorems we use a surgery result obtained by the second author together with Ammann and Dahl \cite{ammann.dahl.humbert:13} and a variational characterization of the mass $m(g,p)$ obtained by the two authors of the present article \cite{hermann.humbert:16}. \\

\noindent {\bf Acknowledgement:} E. Humbert is supported by the project THESPEGE (APR IA), R\'egion Centre-Val de Loire, France, 2018-2020.

\section{Notation} 

Let $M$ be a closed manifold of dimension $n \geq 3$. 
The set of Riemannian metrics on $M$ will be denoted by $\M$. 
For $g \in \M$, we denote by 
\begin{equation*}
  L_g:=\Delta_g+\frac{n-2}{4(n-1)}\scal_g: \quad C^{\infty}(M)\to C^{\infty}(M)
\end{equation*}
the conformal Laplace operator of $g$, where $\scal_g$ is the scalar curvature of $g$ and $\Delta_g$ is the Laplace-Beltrami operator with non-negative spectrum. 
Moreover, we write $N:=\frac{2n}{n-2}$ and we denote by 
\begin{equation*}
  Y(g):=Y(M,g)=\inf \Big\{\frac{\int_M u L_g u\,dv^g}{\|u\|_{L^N}^2}\, \Big| \,u\in C^{\infty}(M)\setminus\{0\}\Big\}
\end{equation*}
the (conformal) Yamabe constant of $g$ and by $\sigma(M):=\sup_{g \in \M} Y(g)$ the (smooth) Yamabe invariant of $M$. 
We will write $Y(g)$ instead of $Y(M,g)$ since $M$ will always be clear from the context.
We have $Y(g)>0$ if and only if all eigenvalues of $L_g$ are strictly positive.
In the following, we will always assume that $\sigma(M)>0.$ 
We define, for any $a \in [0,\sigma(M)[$,
\begin{equation*}
  Z_M(a) := \left\{ g \in \M \mid  Y(g) > a \right\}
\end{equation*}
\begin{equation*}
  \Om^a_M 
  := \left\{ (g , p) \in Z_M(a) \times M \mid  B^g_p(1) \hbox{ is isometric to } \; \mathbb{B} \right\}. 
\end{equation*}
where $B^g_p(1)$ stands for the ball with center $p$ and radius $1$ with respect to the metric $g$ and where $\mathbb{B}$ is the standard Euclidean unit ball of dimension $n$.  
Note that these sets are not empty as soon as $\sigma(M) >0$ (see the first item of Proposition \ref{basic}).   \\

\noindent Let $\eta$ be a smooth function on $M$ such that $\eta\equiv\frac{1}{(n-2)\omega_{n-1}}$ on $B^g_p(\frac{1}{2})$ and $\supp(\eta)\subset B^g_p(1)$, where $\omega_{n-1}$ denotes the volume of $S^{n-1}$ with the standard metric. 
If $(g,p)\in \Om^a_M$ for some $a\geq 0$ then there exists a unique Green function $G_g$ of $L_g$ at $p$ and we have for all $x\in M\setminus\{p\}$:
\begin{equation*}
  G_g(x) = \eta(x)r(x)^{2-n}  + m(g,p) + \alpha(x),
\end{equation*}
where $r(x)$ denotes the Riemannian distance of $x$ and $p$ with respect to~$g$, $\alpha$ is a smooth function defined on all of $M$ which is harmonic on $B^g_p(\frac{1}{2})$ and satisfies $\alpha(p)=0$ and $m(g,p)\in\mR$ is a number called the mass of~$g$ at~$p$.

\noindent We recall that $m(g,p)$ has a variational characterization established in \cite{hermann.humbert:16}. 
Namely, the function $F_{\eta}$: $M\to\mR$ defined by 
\begin{equation*}
F_{\eta}(x):=
\left\{\begin{array}{ll}
\Delta_g(\eta r^{2-n})(x),&x\neq p\\
0, & x=p
\end{array}\right.
\end{equation*}
is smooth on $M$. 
For every $u\in C^{\infty}(M)$ we define 
\begin{equation*}
J^g_p(u):=\int_{M\setminus\{p\}}
\eta r^{2-n}F_{\eta}\,dv^g
+2\int_{M}
uF_{\eta}\,dv^g
+\int_{M}uL_g u\,dv^g.
\end{equation*}
Then, it was proven in \cite{hermann.humbert:16} that 
\begin{equation} \label{vcmass}
- m(g,p)= \inf \{J^g_p(u)\mid u \in C^{\infty}(M)\}
\end{equation}
and that the infimum is attained for the smooth function $\beta$ defined by 
\begin{equation*}
  \beta(x):=m(g,p)+\alpha(x).
\end{equation*}
We say that a closed manifold satisfies PMT (for Positive Mass Theorem) if for every metric $g$ on $M$ and for all points $p\in M$ such that $g$ is flat on an open neighborhood of $p$ and $Y(g)>0$ we have $m(g,p)\geq 0$.
It is conjectured that every closed manifold satisfies PMT.
This conjecture has been proved in some special cases (see e.\,g.\,\cite{schoen.yau:79}, \cite{witten:81}, \cite{schoen.yau:88}).
A complete proof has been announced by Lohkamp \cite{lohkamp:06} and Schoen-Yau \cite{schoen.yau:17}.

\section{Upper and lower mass functions of $M$}
\begin{definition} 
\label{def:mass_functions}
 The {\em upper} (resp. {\em lower}) mass function $X^M_+:[0,\sigma(M)] \to \mR \cup \{\pm \infty \}$  (resp. 
  $X^M_-:[0,\sigma(M)] \to \mR \cup \{\pm \infty \}$) are defined by : for all $a \in [0,\sigma(M)]$
$$\xp{M}{a}:= \limsup_{\ep \to 0} \sup_{(g,p) \in \Om^{\max(a-\ep,0)}_M} m(g,p)$$
(resp.  
$$\xm{M}{a}:= \liminf_{\ep \to 0} \inf_{(g,p) \in \Om^{\max(a-\ep,0)}_M} m(g,p).)$$
\end{definition}
\noindent Note that the maximum in the definitions above is only to ensure that  $\xm{M}{a}$ and $\xp{M}{a}$ are well defined when $a=0$. If $a>0$, one can just replace $\Om^{\max(a-\ep,0)}_M$ by $\Om^{a-\ep}_M$ in these defintions. The goal of this paper is to establish several properties of $\xpm{M}{a}$.

\section{Basic properties of $X^M_\pm$} \label{basic_section}
\begin{proposition} \label{basic}
It holds that 
 \begin{enumerate}
 \item $X^M_+$ and $X^M_-$ are well defined for all $a \in [0,\sigma(M)]$ as soon as $\sigma(M)>0$ ;
 \item For all $a  \in [0,\sigma(M)]$, $\xp{M}{a} \geq \xm{M}{a}$ ;
 \item $X^M_+$ is a decreasing function of $a$ and $X^M_-$ is an increasing function of $a$ and they are both left continuous;
 \item For all $a \in [0,\sigma(M)]$, $0 \geq \xm{M}{a} > - \infty$ ;
 \item For all $a>0$, $\xp{M}{a} < +\infty$ ;
 \item $\xm{M}{0} = 0$ if and only if $M$ possesses the property PMT ;
 \item Let $M,N$ be compact manifolds of dimension $n \geq 3$  with positive Yamabe invariant. 
 Then, for all $a>0$ we have
 \begin{equation*}
   \xp{M \amalg N}{a} = \max(\xp{M}{a},\xp{N}{a}) \; \hbox{ and } \; \xm{M \amalg N}{a} = \min(\xm{M}{a},\xm{N}{a}).
 \end{equation*}
  \end{enumerate}
\end{proposition}

\begin{proof}
\noindent $(1)$ It suffices to show that $\Om^a(M)$ is not empty if $ 0<a <  \sigma(M)$. 
We fix $a' \in (a,\sigma(M))$. First, it is clear that there exists a metric $g$ with $Y(g)=a'$. Hence, let $\xi= \sum_i dx_i^2$, where $(x_1,\cdots,x^n)$ is a system of normal coordinates at some $p\in M$, be a flat metric around $p$ and let 
$g_\ep := (1-\eta_\ep) g + \eta_\ep \xi$, where $\eta_\ep:M \to [0,1]$ is a cut-off function equal to $1$ on $B_p^g(\ep)$, equal to $0$ outside $B_p^g(2\ep)$ and such that $|d \eta_\ep| \leq \frac{2}{\ep}$ and $|\nabla^2\eta_\ep|\leq\frac{2}{\ep^2}$. 
By Lemma \ref{appendixlem} we have $\lim_{\ep\to 0} Y(g_\ep)=Y(g)$.
Now, the metric $h_\ep= \frac{1}{\ep^2}g_\ep$ is flat on $B_p^{h_\ep}(1)$. If $\ep$ is small enough then $Y(h_\ep) = Y(g_\ep) >a$ which implies that $h_\ep \in \Om^a(M)$. \\

\noindent $(2)$ and $(3)$ are clear from the definitions. \\

\noindent $(4)$ Let $u$ be any nonzero smooth function compactly supported in the Euclidean ball $\mathbb{B}$. Let $(g,p) \in \Om_M^a$ for some $a$. From the definition of $\Om_M^a$, we can identify $(B_p(1),g)$ with $\mathbb{B}$ so that $u$ can be considered as a test function in the 
  variational characterization (\ref{vcmass}) which provides
  $$-m(g,p) \leq J^p_g(u).$$
  The inequality $\xm{M}{a} > - \infty$ follows by noticing that $J^p_g(u)$ does not depend on the choice of $a \geq 0$ nor on the choice of $(g,p) \in \Om_M^a$.\\
  
\noindent Let us prove now that $\xm{M}{a} \leq 0$.   
It comes from the facts that if $(g,p) \in  \Om_M^a$ then $(bg,p) \in  \Om_M^a$ for any $b \geq 1$, and also that for any such $b$ 
  $$m(b g,p) = b^{1-\frac{n}{2}}m(g,p). $$

 \noindent $(5)$ 
   Let t$a>0$ and let $(g,p) \in  \Om^{a}_M$. We have to show that $m(g,p)$ is bounded by a constant which depends only on $a$ but not on $(g,p)$.  Let $u \in C^{\infty}(M)$. In what follows, $C>0$ denotes a positive constant which might depend on $a$ but not on $(g,p)$. 
  By the variational characterization \eqref{vcmass}, choose $u$ so that  
  $$-m(g,p) + 1 \geq  J_{g}^{p}(u)$$
  From the definition of $J_g^p$, one has
 $$J_{g}^{p}(u) \geq - C + 2\int_{M}
u F_{\eta}\,dv^{g}
+\int_{M}u  L_{g} u\,dv^{g}.$$
Using that fact that $Y(g) \geq a$ and using  H\"older inequality, one gets 
$$J_{g}^{p}(u) \geq -C -2 \| F_\eta\|_{L^{\infty}} \left( \int_{B_p^g(1)}  |u|^N dv^{g} \right)^{\frac{1}{N}} \vol(B_p^g(1))^{\frac{N-1}{N}}
+a \left(\int_{M}|u|^N,dv^{g} \right)^{\frac{2}{N}}.$$
Set now
$$X_g=  \left( \int_M |u|^N dv^{g} \right)^{\frac{1}{N}},$$
we obtain that there exists some $C',C''>0$  independent of $(g,p)$ such that 
\begin{equation} \label{Jgeq} 
 J_{g}^{p}(u) \geq C-C' X_g +C''X_g^2.
\end{equation}

\noindent This quantity is bounded from below independently of $(g,p)$. This show that $m(g,p)$ is bounded from above by a constant independent of $(g,p) \in \Om^a_M$. This implies that for all $a>0$, $\xp{M}{a} < +\infty$. \\

\noindent $(6)$ Clearly the property PMT for a manifold is equivalent to $X^M_-(0) \geq 0$. Since $X^M_-(0) \leq 0$ by item $(4)$, the result follows.\\

 \noindent $(7)$ Let $a \in (0, \sigma(M \amalg N)] = (0 , \min (\sigma(M), \sigma(N)) ]$ and  $\ep > 0$. On the one hand, let $(g,p) \in \Om^{a-\ep}_{M \amalg N}$ where $p \in M$.  Then  $g$ decomposes as 
 $g= g_M \amalg g_N$ where $g_M \in \M$ and $g_N \in \mathcal{M}_N$. We have $a-\ep < Y(g)= \min(Y(g_M),Y(g_N))$. Since $p \in M$, this implies that $(g_M,p) \in \Om^{a-\ep}_M$. 
 Let $\rho>0$. If $\ep$ is small enough, it follows from the definition of $\xp{M}{a}$ that 
 $m(g,p)=m(g_M,p) \leq X^M_+(a)+\rho \leq \max( X^M_+(a), X^N_+(a)) + \rho$.
 In the same way, if $p \in N$, $m(g,p) \leq \max( X^M_+(a), X^N_+(a))+ \rho $. From these inequalities, we obtain 
 $$X^{M \amalg N}_+(a) \leq  \max( X^M_+(a), X^N_+(a)) + \rho$$
 and since $\rho$ is arbitrary
 $$X^{M \amalg N}_+(a) \leq  \max( X^M_+(a), X^N_+(a)).$$
 On the other hand, let $(g_M,p) \in \Om^{a-\ep}_M$ and let $g_N$ any metric on $N$ with $Y(g_N) \geq a- \ep$.  If $\ep$ is small enough, then $m(g_M,p) = m(g_M \amalg g_N,p) \leq X^{M \amalg N}_+(a) + \rho$. Hence $X^M_+(a) \leq X^{M \amalg N}_+(a)+ \rho$. The same holds for $N$ and the result follows. \\
 
 \noindent The proof for $\xm{M}{a}$ is similar. 
 
\end{proof}

\section{$\xpm{M}{a}$ and surgery}
\label{sec:surgery}
In this section, we first establish the following theorem, whose proof is a consequence of the results in \cite{hermann.humbert:16} and \cite{ammann.dahl.humbert:13}
\begin{theorem} \label{main_surgery} 
 Let $M$ be a compact manifold of dimension $n\geq 3$ and $M^\sharp$ be obtained from $M$ by a surgery of dimension $k \leq n-3$. Then, for all $a \in [0,\sigma(M)]$, 
one has 
$$\xp{M}{a} \leq \xp{M^\sharp}{\min(a,\La_{n,k})}   \;  \hbox{ and } \;  \xm{M}{a} \geq \xm{M^\sharp}{\min(a,\La_{n,k})}.$$
where 
$\La_{n,0}= + \infty$ and where $\La_{n,k}>0$ depends only on $n$ and $k$. 
 \end{theorem}

\noindent A consequence of Theorem \ref{main_surgery} is 
\begin{corollary} \label{cor_surg}
 Let $M_0$ be any compact non spin (resp. spin) simply connected manifold of dimension $n \geq 5$ such that $\si(M_0)>0$ and let $a>0$. Then, for all compact (resp. compact spin) manifolds $M$ of the same dimension one has 
 \begin{align*}
   0 & \geq X^M_-(\min(a,\La_n)) \geq X^{M_0}_-(\min(a,\La_n))\;\hbox{ and } \\
   & \;\;\;\;\; X^M_+(\min(a,\La_n))  \leq X^{M_0}_+(\min(a,\La_n))
 \end{align*}
 where $\La_n = \min_{1 \leq k \leq n-3} \La_{n,k}>0$. 
\end{corollary}

\noindent This has the following obvious consequence:
\begin{corollary} \label{cor_surg2}
 Let $M_0, M_1$ be two compact non spin (resp. spin) simply connected manifolds of dimension $n \geq 5$ such that $\si(M_0), \sigma(M_1)>0$ and let $a\in(0,\Lambda_n)$. 
 Then we have 
 $$X^{M_0}_{\pm} (a) = X^{M_1}_{\pm} (a).$$
\end{corollary}
\begin{remark}
 By Corollary C in \cite{gromov.lawson:80}, if $M$ is a compact simply connected non-spin manifold of dimension at least $5$ then $\sigma(M)>0$. 
 By \cite{ammann.dahl.humbert:13}, when $M$ is simply connected and $\si(M)>0$, it holds  that
 \begin{equation*}
   \si(M) \geq \min\{\La_n,\sigma(W_1),...,\sigma(W_k)\}
 \end{equation*}
 where $W_1,...,W_k$ are generators of the oriented cobordism group in dimension $n$.
\end{remark}
\begin{remark}
  This corollary allows to recover a result in \cite{hermann.humbert:16}: if $M_0$ not spin, simply connected of dimension $n \geq 5$ satisfies 
  PMT, then all the manifolds of the same dimension satisfy PMT. Indeed, assume that $M_0$ satisfies PMT then $X^{M_0}_-(0) = 0$ (see Proposition \ref{basic}) and hence $X^M_-(0)=0$ which implies PMT. 
  Note that Lohkamp \cite{lohkamp:06} and Schoen and Yau \cite{schoen.yau:17} recently announced a complete proof of the Positive Mass Theorem (i.e. all manifolds satisfy PMT).  

\end{remark}

\noindent Another  consequence is the following: 
\begin{corollary} \label{xp>0}
Assume that $M$ is simply connected, that $\si(M) >0$ and that $a < \La_n$. Then, $\xp{M}{a}>0$.
\end{corollary}
\begin{remark}
Again, if the proof of the Positive Mass Theorem by Lohkamp in \cite{lohkamp:06} or by  Schoen and Yau announced in  \cite{schoen.yau:17} is confirmed then, for all $M$ and all $a < \sigma(M)$,   $\xp{M}{a}>0$.
\end{remark}


\subsection{Proof of Theorem \ref{main_surgery}}
Let $g \in \Om^a_M$. In \cite{hermann.humbert:16}, we constructed a sequence of metrics $g_k$ on $M^\sharp$ such that $\lim_k m(g_k,p)=m(g,p)$. In the  construction, the metric $g_k$ can be made isometric to $g$ in $B_p^g(1)$ (as soon as $B_p^g(1)$ is topologically trivial).   Moreover, we used exactly the same metrics as in the main result of \cite{ammann.dahl.humbert:13} where it was proved that 
$$\lim_k Y(g_k) \geq \min(\La_{n,k},Y(g))$$
where $\La_{n,0}= + \infty$ and $\La_{n,k} >0$ depends only on $n$ and $k$. This proves that for all $\ep >0$ we have $g_k \in \Om^{\min(a,\La_{n,k})-\ep}_{M^\sharp}$ as soon as $k$ is large enough. 
Theorem \ref{main_surgery} easily follows.

\subsection{Proof of Corollary \ref{cor_surg}}
\noindent {$(1)$} Let $M_0$ be a compact non-spin (resp. spin) simply connected manifold with $\si(M_0)>0$ and $M$ any compact (resp. compact spin) manifold of the same dimension. 
By Proposition \ref{basic}, 
$$X^{M \amalg (-M)}_+(a) = X^M_+(a),$$
where $(-M)$ is $M$ equipped with the opposite orientation.
Theorem \ref{main_surgery} then shows that 
\begin{eqnarray} \label{ineq1}
 X^M_+(a) =  X^{M \amalg (-M)}_+(a)  \leq X^{M \sharp (-M)}_+ (\min(a,\La_{n,k}))
\end{eqnarray}
where  $\sharp$ denotes the connected sum. 
Here, we used that the connected sum is a surgery of dimension $0$. \\

\noindent {$(2)$} The manifolds $M \sharp (-M)$ and $M_0 \sharp (-M_0)$ are oriented (resp. spin) cobordant since they are both oriented (resp. spin) cobordant to $S^n$. 
Since $M_0 \sharp (-M_0)$ is simply connected and not spin (resp. spin), it is obtained from $M \sharp (-M)$ by a finite sequence of surgeries of dimension $k \leq n -3$ (see the proofs of Theorem B and Theorem C in the article \cite{gromov.lawson:80} by Gromov-Lawson). 
Theorem \ref{main_surgery} then implies that 
\begin{eqnarray} \label{ineq2}
  X^{M \sharp (-M)}_+(\min(a,\La_{n,k}))  \leq X^{M_0 \sharp (-M_0)}_+(\min(a,\La_{n,k})).
\end{eqnarray}

\noindent Inequality (\ref{ineq1}) remains true when $M$ is replaced by $M_0$. As a consequence, we get from Proposition \ref{basic} that 
$$ X^{M_0 \sharp (-M_0)}_+(a)  = \max (X^{M_0 \sharp (-M_0)}_+ (a) , X^{M_0}_+(a))= X^{M_0 \sharp (-M_0) \amalg M_0}_+(a).$$
Using Theorem \ref{main_surgery}, we obtain 
$$\xp{M_0 \sharp (-M_0)}{\min(a,\La_{n,k})} \leq \xp{M_0 \sharp (-M_0) \sharp M_0}{\min(a,\La_{n,k})}$$
Now, $M_0 \sharp (-M_0) \sharp M_0$ is oriented (resp. spin) cobordant to $M_0$ and $M_0$ is simply connected and not spin (resp. spin):  by the same argument as above, $M_0$ is obtained from $M$ by a finite sequence of surgeries of dimension $k \leq n-3$. This proves that 
$$\xp{M_0 \sharp (-M_0)}{\min(a,\La_{n,k})} \leq \xp{M_0 \sharp (-M_0) \sharp M_0}{\min(a,\La_{n,k})} \leq \xp{M_0}{\min(a,\La_{n,k})}.$$
Together with Inequalities (\ref{ineq1}) and (\ref{ineq2}), we obtain the desired inequality
 $$\xp{M}{\min(a,\La_{n,k})} \leq \xp{M_0}{\min(a,\La_{n,k})}.$$ 
 
\noindent {$(3)$} The argument for $\xm{M}{a}$ is similar.

\section{Application to the Yamabe invariant}

\begin{theorem} \label{thm:xp_sigma} $\;$ \\

1. For any compact manifold $M$ with $\sigma(M)>0$, one has $\xp{M}{0}= + \infty$.\\

2. For every $n\geq 3$ there exists a constant $d_n>0$ such that for all compact manifolds $M$ of dimension $n$ with $\sigma(M)>0$ and for all $a\in(0,\sigma(M))$ we have
  \begin{equation*}
    \xp{M}{a} \leq d_n\,\frac{(\sigma(S^n)-a)^{1/n}}{a}.
  \end{equation*}
\end{theorem}

\begin{cor}\label{cor:xp_sigma}
  Let $d_n$ be the constant in part 1 of Theorem \ref{thm:xp_sigma} and suppose that $M$ is a compact manifold of dimension $n$ with $\sigma(M)>0$ such that
  \begin{equation*}
    \frac{d_n}{\sigma(S^n)}
    < \limsup_{\ep\to 0} \frac{\xp{M}{\sigma(M)-\ep}}{\ep^{1/n}}.
  \end{equation*}
  Then we have $\sigma(M)<\sigma(S^n)$.
\end{cor}

Note that the hypothesis of Corollary \ref{cor:xp_sigma} is satisfied if $\xp{M}{\sigma(M)}>0$ since the function $a\mapsto\xp{M}{a}$ is continuous from the left. 
This fact leads to a natural question: is this possible that $X^M_+(\sigma(M)) >0$ ?  The answer is given by
\begin{proposition} \label{examples}
 It holds that
 $$  X^{\mR P^3}_+(\sigma(\mR P^3)) >0 \; \hbox{ and } \; X^{S^n}_+(\sigma(S^n)) =0.$$
 \end{proposition}

 \begin{proof}[Proof of Proposition \ref{examples}]
The fact that $X^{S^n}_+(\sigma(S^n)) =0$ is an immediate consequence of Corollary \ref{cor:xp_sigma}.  Beside, it was proven by Bray and Neves \cite{bray.neves:04}  that $\sigma(\mR P^3)< \sigma(S^3)$ is attained by the standard metric. Since the standard metric of $\mR P^3$ is locally conformally flat, one can choose a metric $g$ in its conformal class such that $B_p^{g}(1)$ is flat where $p \in \mR P^3$ is fixed. Then, since $\mR P^3$ satisfies $PMT$, one has
$$ X^{\mR P^3}_+ (\sigma(\mR P ^3)) \geq m(g,p) >0.$$ 
\end{proof}

\begin{proof}[Proof of Corollary \ref{cor:xp_sigma}]
Assume that $\sigma(M)=\sigma(S^n)$.
Using Theorem \ref{thm:xp_sigma} we get 
\begin{equation*}
  \frac{d_n}{\sigma(S^n)}
  < \limsup_{\ep\to 0}  \frac{d_n\,\ep^{1/n}}{(\sigma(S^n)-\ep)\ep^{1/n}}
  = \limsup_{\ep\to 0} \frac{d_n}{\sigma(S^n)-\ep}
\end{equation*}
which is a contradiction.
\end{proof}

\begin{proof}[Proof of Theorem \ref{thm:xp_sigma}]
For the first statement, let $p \in M$ and $g_m \in \Om^0_M$ be a sequence of metrics converging in $C^2$ to some metric $g_\infty$ such that $Y(g_\infty) = 0$. To see the existence of such a sequence, it suffices to construct $g_\infty$. For this, just consider any metric $(g,p) \in \Om^a_M$ for some $a>0$. It is standard that  one can modify $g$ locally outside $B_p^g(1)$ to get a metric $h$ such that $Y(h)<0$. Then, set $g_t= t g+(1-t) h$ for $t \in [0,1]$. 
 Let $$t_\infty:= \max \{ t\in[0,1] \mid Y(g_t) \leq 0 \}.$$
We can then set $g_\infty:=g_{t_\infty}$ and $g_m :=g_{t_\infty + \frac{1}{m}}$.  It was then  proven in \cite{hermann.humbert:16} and in \cite{beig.omurchada:97} that $\lim_m m(g_m,p)= + \infty$ which proves that 
  $X^M_+(0)=+\infty$.\\
  
 The second statement is much harder to prove.  
   Let $a\in(0,\sigma(M))$ and let $(g_\ep)_{\ep\in(0,a)}$ be a  sequence of Riemannian metrics on $M$ such that for all $\ep$ we have $g_\ep\in\Omega^{a-\ep}_M$ 
and 
\begin{equation*}
  A_\ep:=m(g_\ep,p)\geq \xp{M}{a}+\ep.
\end{equation*}
For every $\ep$ let $\beta_\ep$ be the smooth function on $M$ such that $-A_\ep=J^g_p(\beta_\ep)$.
We put 
\begin{equation*}
X_\ep:=\Big(\int_M |\beta_\ep|^N\,dv^{g_\ep}\Big)^{\frac{1}{N}}.
\end{equation*}
By H\"older's inequality and by the definition of $Y(g_\ep)$ we have 
\begin{align*}
  -A_\ep 
  &=
  J^{g_\ep}_p(\beta_\ep)
  = \int_{M\setminus\{p\}} \eta r^{2-n}F_{\eta}\,dv^{g_\ep}
  +2\int_M \beta F_{\eta}\, dv^{g_\ep}
  +\int_M \beta L_g \beta\,dv^{g_\ep} \\
  &\geq \int_{M\setminus\{p\}} \eta r^{2-n}F_{\eta}\,dv^{g_\ep}
  -2 \Big(\int_M |F_{\eta}|^{\frac{N}{N-1}}\, dv^{g_\ep} \Big)^{\frac{N-1}{N}} X_\ep
  +(a-\ep)
  X_\ep^2.
\end{align*}
We put 
\begin{equation*}
  C_n:= \int_{M\setminus\{p\}} \eta r^{2-n}F_{\eta}\,dv^g
  ,\quad 
  D_n:= 2 \Big(\int_M |F_{\eta}|^{\frac{N}{N-1}}\, dv^g\Big)^{\frac{N-1}{N}}.
\end{equation*}
The numbers $C_n$, $D_n$ are independent of $(M,g)$ since $g$ is flat on the supports of $\eta$ and $F_{\eta}$.
We get 
\begin{equation*}
-A_\ep \geq C_n - D_nX_\ep + (a-\ep) X_\ep^2.
\end{equation*}
We have $C_n\geq 0$ since for $M=S^n$ with the standard metric $g_{\mathrm{can}}$ we have 
\begin{equation*}
  0 = -m (g_{\mathrm{can}},p)
  \leq J^{g_{\mathrm{can}}}_p(0)
  =C_n.
\end{equation*}
We also have $C_n+A_\ep\geq 0$ for all $\ep$ since 
\begin{equation*}
  -A_\ep 
  =m(g_\ep,p)
  \leq J^{g_\ep}_p(0)
  =C_n.
\end{equation*}
The quadratic function $f(x):=A_\ep+C_n-D_nx+(a-\ep)x^2$ satisfies $f(X_\ep)\leq 0$ and attains its minimum value at $x_0=-\frac{D_n}{2(a-\ep)}$.
We get 
\begin{equation*}
  0 \geq f(x_0)
  = A_\ep + C_n - \frac{D_n^2}{4(a-\ep)}
\end{equation*}
and thus 
\begin{equation*}
  \xp{M}{a}+\ep 
  \leq A_\ep 
  \leq \frac{D_n^2}{4(a-\ep)}-C_n
  \leq \frac{D_n^2}{4(a-\ep)}.
\end{equation*}
As $\ep\to 0$ we obtain
\begin{equation}
  \label{eq:xp_leq_C/a}
  \xp{M}{a}\leq \frac{D_n^2}{4a}
\end{equation}
Moreover, since $f(X_\ep)\leq 0$, the number $X_\ep$ is less than or equal to the largest root of the equation $f(x)=0$.
Using that $A_\ep+C_n\geq 0$ we get 
\begin{equation*}
  X_\ep 
  \leq \frac{D_n+\sqrt{D_n^2-4(A_\ep+C_n)(a-\ep)}}{2(a-\ep)}
  \leq \frac{D_n}{a-\ep}.
\end{equation*}
Recall that the function $\beta_\ep$ is harmonic on $B^{g_\ep}_p(\frac{1}{2})$. 
Thus for all $x\in B^{g_\ep}_p(\frac{1}{3})$ we have 
\begin{equation*}
  \beta_\ep(x)
  =\frac{6^n}{\vol(\mathbb{B})} \int_{B^{g_\ep}_x(\frac{1}{6})} \beta_\ep (y)\, dv^{g_\ep},
\end{equation*}
where $\mathbb{B}$ denotes the Euclidean unit ball, since $B^{g_\ep}_x(\frac{1}{6})\subset B^{g_\ep}_p(\frac{1}{2})$.
Using H\"older's inequality we get for all $x\in B^{g_\ep}_p(\frac{1}{3})$:
\begin{align}
  \nonumber 
  |\beta_\ep(x)|
  &\leq \frac{6^n}{\vol(\mathbb{B})}
  \vol\Big(B^{g_\ep}_x\Big(\frac{1}{6}\Big)\Big)^{\frac{N-1}{N}}\,\Big(\int_{B^{g_\ep}_x(\frac{1}{6})}|\beta_\ep|^N\Big)^{\frac{1}{N}}\\
  \nonumber 
  &\leq 6^{\frac{n}{N}} \vol(\mathbb{B})^{-\frac{1}{N}}\, X_\ep \\
  \nonumber 
  &\leq 6^{\frac{n}{N}} \vol(\mathbb{B})^{-\frac{1}{N}}\, \frac{D_n}{a-\ep}\\
  \label{eq:beta_estimate}
  &=: \frac{B_n}{a-\ep}.
\end{align}
For every $\ep$ we define $\rho_\ep>0$ such that 
\begin{equation}
  \label{eq:rho_ep_definition}
  \rho_\ep^{n-2} = \delta_n  \frac{|A_\ep|}{(B_n+|A_\ep|)^2}
\end{equation}
where $\delta_n>0$ can be chosen such that for all $\ep$ the number $\rho_\ep$ is as small as we want since the function $x\mapsto\frac{|x|}{(B_n+|x|)^2}$ is bounded.
We choose $\delta_n$ such that  $\rho_\ep<\frac{1}{6}$ for all $\ep$. 
Then for every $\ep$ we choose $h_\ep\in C^{\infty}(M)$ such that $0\leq h_\ep\leq 1$, $h_\ep\equiv 1$ on $B^{g_\ep}_p(\rho_\ep)$, $h_\ep\equiv 0$ on $M\setminus B^{g_\ep}_p(2\rho_\ep)$ and $|dh_\ep|\leq\frac{2}{\rho_\ep}$.
Moreover, for every $\ep$ we write the Green function $G_\ep$ of $L_{g_\ep}$ as
\begin{equation*}
  G_\ep(x)
  =\eta(x)r(x)^{2-n}
  + A_\ep 
  + \alpha_\ep (x)
\end{equation*}
where $\alpha_\ep\in C^{\infty}(M)$ is harmonic on $B^{g_\ep}_p(\frac{1}{2})$ and satisfies $\alpha_\ep(p)=0$.

\begin{step} \label{step1_new}
For $\ep$ close enough to $0$ we have 
\begin{equation*}
  \int_M  |d  (h_\ep \al_\ep)|^2 dv^{g_\ep}  \leq 
  \frac{(n-2)\omega_{n-1}}{4}\,|A_\ep|.
\end{equation*}
\end{step}

Since $\alpha_\ep$ is harmonic on the support of $d h_\ep$, we can use  Identity $(3)$ in \cite{hermann.humbert:16} and  H\"older inequality
to write 
\begin{equation*}
  \int_M  |d  (h_\ep \al_\ep)|^2 dv^{g_\ep}  
  = \int_{C_\ep} |dh_\ep|^2 \al_\ep^2 dv^{g_\ep}
  \leq \left( \int_{C_\ep} |d h_\ep|^n dv_{g_\ep} \right)^{\frac{2}{n}} \left( \int_{C_\ep} |\alpha_\ep|^N dv_{g_\ep} \right)^{\frac{2}{N}}
\end{equation*}
where $C_\ep:= B_{p}^{g_\ep}(2 \rho_\ep) \setminus B_{p}^{g_\ep}(\rho_\ep)$ is the support of $dh_\ep$. 
Observe that the definition of $h_\ep$ and the fact that the volume of the support of $dh_\ep$ is bounded by $C \rho_\ep^n$ with $C$ independent of $\ep$ imply that there exists $\alpha_0>0$ which is independent of $\ep$ such that 
\begin{equation*}
  \left( \int_{C_\ep} |d h_\ep|^n dv_{g_\ep} \right)^{\frac{2}{n}} 
  \leq \al_0.
\end{equation*}
Hence, for all $\ep$ small enough, 
\begin{equation} \label{eq:ineq_alpha_0}
  \int_M  |d  (h_\ep \al_\ep)|^2 dv^{g_\ep}  
  \leq \al_0  \left( \int_{C_\ep} |\alpha_\ep|^N dv_{g_\ep} \right)^{\frac{2}{N}}.
\end{equation}
Now, with $\alpha_\ep=\beta_\ep-A_\ep$ and using the equations \eqref{eq:beta_estimate} and \eqref{eq:rho_ep_definition} we get for $\ep$ close to $0$:
\begin{align*}
 \Big(\int_{C_\ep} |\alpha_\ep|^N\,dv^{g_\ep}\Big)^{\frac{1}{N}}
 &\leq \Big(\int_{C_\ep} |\beta_\ep|^N\,dv^{g_\ep}\Big)^{\frac{1}{N}}
 + |A_\ep| \vol(C_\ep)^{\frac{1}{N}}\\
 &\leq \Big(\frac{B_n}{a-\ep}+|A_\ep|\Big)\vol(C_\ep)^{\frac{1}{N}}\\
 &=\Big(\frac{B_n}{a-\ep}+|A_\ep|\Big)(2^n-1)^{\frac{1}{N}}\rho_\ep^{\frac{n-2}{2}} \vol(\mathbb{B})^{\frac{1}{N}}\\
 &\leq \Big(\frac{B_n}{a-\ep}+|A_\ep|\Big)(2^n-1)^{\frac{1}{N}} \delta_n^{\frac{n-2}{2}} \frac{\sqrt{|A_\ep|}}{B_n+|A_\ep|} \vol(\mathbb{B})^{\frac{1}{N}}\\
 &\leq E_n\sqrt{|A_\ep|}
\end{align*}
where $E_n>0$ is independent of $\ep$.
By choosing $\delta_n$ in equation \eqref{eq:rho_ep_definition} smaller we may assume that $E_n^2\alpha_0\leq \frac{(n-2)\omega_{n-1}}{4}$.
Therefore the assertion of Step 1 follows from the equation \eqref{eq:ineq_alpha_0}.\\


\noindent 
For every $\ep$ we define
\begin{equation}
  \label{eq:beta_definition}
  \beta_\ep := \zeta_n |A_\ep|^{\frac{1}{2}} \rho_\ep^{\frac{n}{2}}
\end{equation}
where $\zeta_n>0$ will be fixed later. 
We have for all $\ep$:
\begin{equation}
  \label{eq:rho_ep_beta_ep}
  \be_\ep 
  = \zeta_n |A_\ep|^{\frac{1}{2}} \rho_\ep^{\frac{n-2}{2}} \rho_\ep
  = \zeta_n \delta_n^{\frac{1}{2}} \frac{|A_\ep|}{B_n+|A_\ep|} \rho_\ep
  \leq \zeta_n \delta_n^{\frac{1}{2}} \rho_\ep.
\end{equation}
We define $u_\ep\in C^{\infty}(M)$ by 
\begin{equation*}
  u_\ep(r) = \left( \frac{\beta_\ep}{\beta_\ep^2 + r^2} \right)^{\frac{n-2}{2}}
\end{equation*}
and $\psi_\ep \in C^{\infty}(M)$ by 
\begin{equation*}
  \psi_\ep = \left| \begin{array}{ccc}
                    u_\ep & \hbox{ if } & r \leq \rho_\ep \\
                    \ell_\ep( G_\ep - h_\ep \al_\ep) & \hbox{ if } & \rho_\ep \leq r \leq 2 \rho_\ep \\
                    \ell_\ep G_\ep & \hbox{ if } & r \geq 2 \rho_\ep 
                    \end{array}
  \right.
\end{equation*}
where $\ell_\ep= u_\ep(\rho_\ep)(\frac{1}{(n-2)\omega_{n-1}}\rho_\ep^{2-n} +A_\ep)^{-1}$ so that $\psi_\ep$ is continuous. 

\begin{step} \label{step2_new}
  Conclusion. 
\end{step}


\noindent We set  
\begin{equation*}
  E_\ep 
  = \int_M \Big( |d \psi_\ep|^2 + \frac{n-2}{4(n-1)}\scal_{g_\ep} |\psi_\ep|^2 \Big) dv^{g_\ep}
\end{equation*}
and 
\begin{equation*}
  D_\ep 
  = \left( \int_M |\psi_\ep|^N\, dv^{g_\ep} \right)^{\frac{2}{N}}
\end{equation*}
so that 
\begin{eqnarray} \label{eq:Q=}
 Q_\ep(\psi_\ep) = \frac{E_\ep}{D_\ep}.
\end{eqnarray}

\noindent We write 
\begin{equation*}
  E_\ep = E_1 + E_2
\end{equation*}
where 
\begin{equation*}
  E_1 = \int_{B_\ep} \Big( |d \psi_\ep|^2 + \frac{n-2}{4(n-1)} \scal_{g_\ep} |\psi_\ep|^2 \Big) dv^{g_\ep}
\end{equation*}
and 
\begin{equation*}
  E_2 = \int_{M \setminus B_\ep} \Big( |d \psi_\ep|^2 +  \frac{n-2}{4(n-1)} \scal_{g_\ep} |\psi_\ep|^2 \Big) dv^{g_\ep}
\end{equation*}
where $B_\ep:=B_{p}^{g_\ep}(\rho_\ep)$ which is isometric to the Euclidean ball of radius $\rho_\ep$. 
On $B_\ep$, it holds that 
\begin{equation*}
  \Delta_{g_\ep} u_\ep = n(n-2) u_\ep^{N-1}
\end{equation*}
and we get from multiplying this equation by $u_\ep$ and   integrating by parts that 
\begin{equation}
  \label{eq:ubeta}
  \int_{B_\ep} |du_\ep|^2 dv^{g_\ep} 
  -\int_{\partial B_\ep} u_\ep \frac{\partial u_\ep}{\partial r} \,da^{g_\ep}
  = n(n-2) \int_{B_\ep} |u_\ep|^N dv^{g_\ep}.
\end{equation}
One also has 
\begin{align*} 
  \sigma(S^n) & = \frac{\int_{\mR^n} |du_\ep|^2 dx}{\left( \int_{\mR^n} |u_\ep|^N dx \right)^{\frac{2}{N}}}  
  = n (n-2)  \left( \int_{\mR^n} |u_\ep|^N dx \right)^{\frac{2}{n}}
\end{align*}
which leads to 
\begin{equation*}
  \left( \int_{B_\ep} |u_\ep|^N dx \right)^{\frac{2}{n}} 
  \leq \frac{\sigma(S^n)}{n(n-2)}
\end{equation*}
Plugging this estimate into equation \eqref{eq:ubeta}, we obtain 
\begin{equation} \label{eq:E1}
  E_1 
  \leq \sigma(S^n) \left( \int_{B_\ep} |\psi_\ep|^N dv^{g_\ep} \right)^{\frac{2}{N}} 
  + \int_{\partial B_\ep} u_\ep \frac{\partial u_\ep}{\partial r}\,da^{g_\ep}.
\end{equation}
Now, we evaluate $E_2$. For this, we integrate by parts: 
\begin{equation*}
  E_2= \int_{M \setminus B_\ep} \psi_\ep L_{g_\ep}\psi_\ep\, dv_{g_\ep}  
  +E_3
\end{equation*}
where $E_3$ is a boundary term which will be computed below. Since $L_{g_\ep} G_\ep= 0 $ on $M \setminus B_\ep$, 
we obtain 
\begin{equation*}
  E_2= \ell_\ep^2 \int_{B_\ep' \setminus B_\ep} L_{g_\ep} (- h_\ep \al_\ep) (G_\ep - h_\ep \al_\ep) dv_{g_\ep}  
  + E_3
\end{equation*}
where $B':= B_{p}^{g_\ep}(2 \rho_\ep)$. 
Note that since $\al_\ep$ is harmonic and $h_\ep$ is constant on $B_\ep$, one has on $B_\ep$,
\begin{equation*}
  L_{g_\ep} (- h_\ep \al_\ep) = \Delta_{g_\ep}  (- h_\ep \al_\ep) = 0.
\end{equation*} 
Hence, by definition of the Green function $G_\ep$, one has: 
\begin{equation*}
  \int_{B_\ep' \setminus B_\ep} L_{g_\ep} (- h_\ep \al_\ep)G_\ep dv^{g_\ep} 
  = \int_M  L_{g_\ep} (- h_\ep \al_\ep)G_\ep dv^{g_\ep} 
  = -\al_\ep(p) = 0.
\end{equation*}
We also have
\begin{equation*}
  \int_{B_\ep' \setminus B_\ep}  L_{g_\ep} (h_\ep \al_\ep)(h_\ep \al_\ep)  dv^{g_\ep} 
  =  \int_M  L_{g_\ep} (h_\ep \al_\ep)(h_\ep \al_\ep) 
  =  \int_M  |d  (h_\ep \al_\ep)|^2 dv^{g_\ep}.
\end{equation*}
Since, by Step \ref{step1_new} we have
\begin{equation}
  \label{eq:gamma_definition}
  \int_M  |d (h_\ep \al_\ep)|^2 dv^{g_\ep} \leq \frac{(n-2)\omega_{n-1}}{4}|A_\ep|
  =:\gamma_\ep
\end{equation} 
we obtain:
\begin{equation*}
  E_2 \leq \ell_\ep^2 \gamma_\ep + E_3.
\end{equation*}
So let us evaluate $E_3$:
\begin{align*}
  E_3 & = - \ell_\ep^2   \int_{\partial B_\ep} (G_\ep-h_\ep \al_\ep) 
  \frac{ \partial ( G_\ep - h_\ep \al_\ep)}{ \partial r}\\ 
  & = -\ell_\ep^2 \Big(\frac{1}{(n-2)\omega_{n-1}}\rho_\ep^{2-n}+A_\ep\Big)\Big(-\frac{1}{\omega_{n-1}}\rho_\ep^{1-n}\Big)\vol(\partial B_\ep)\\
  & = \ell_\ep^2 \Big(\frac{1}{(n-2)\omega_{n-1}}\rho_\ep^{2-n} + A_\ep\Big). 
\end{align*}
Combining this with \eqref{eq:E1}
\begin{align}
  \nonumber  
  E_1 + E_2  
  &\leq \sigma(S^n) \left( \int_{B_\ep} |\psi_\ep|^N dv^{g_\ep} \right)^{\frac{2}{N}}   \\
  \label{eq:E_1+E_2}
  & + \int_{\partial B_\ep} u_\ep \frac{\partial u_\ep}{\partial r} da^{g_\ep} +  \ell_\ep^2 \Big(\frac{1}{(n-2)\omega_{n-1}}\rho_\ep^{2-n} + A_\ep\Big) +\ell_\ep^2 \gamma_\ep.
\end{align}

\noindent  It remains to compute 
\begin{equation*}
  E_4:= \int_{\partial B_\ep}  u_\ep \frac{\partial u_\ep}{\partial r} da^{g_\ep}.
\end{equation*}
Using the fact that on $\partial B_\ep$ we have
\begin{equation*}
  u_\ep= \ell_\ep \Big( \frac{1}{(n-2)\omega_{n-1}}\rho_\ep^{2-n} +A_\ep\Big)
\end{equation*} 
and that 
\begin{equation*}
  \frac{\partial u_\ep}{\partial r} = -(n-2) \frac{\rho_\ep}{\be_\ep^2+\rho_\ep^2} u_\ep
\end{equation*}
we obtain 
\begin{equation*}
  E_4 = -(n-2) \om_{n-1}  \frac{\rho_\ep^n}{\be_\ep^2+\rho_\ep^2} \ell_\ep^2 \Big( \frac{1}{(n-2)\omega_{n-1}}\rho_\ep^{2-n} +A_\ep\Big)^2
\end{equation*}
which together with the definition \eqref{eq:beta_definition} of $\beta_\ep$ leads to 
\begin{align*}
  &\quad \int_{\partial B_\ep} u_\ep \frac{\partial u_\ep}{\partial r} da^{g_\ep} +  \ell_\ep^2 \Big(\frac{1}{(n-2)\omega_{n-1}}\rho_\ep^{2-n} + A_\ep\Big)\\
  &= \ell_\ep^2 \Big( \frac{1}{(n-2)\omega_{n-1}}\rho_\ep^{2-n} +A_\ep\Big)
  \Big( 1 - \frac{\rho_\ep^n}{\be_\ep^2+\rho_\ep^2}(\rho_\ep^{2-n}+(n-2)\omega_{n-1}A_\ep)\Big)\\
  &= \ell_\ep^2 \Big( \frac{1}{(n-2)\omega_{n-1}}\rho_\ep^{2-n} +A_\ep\Big)
  \frac{\be_\ep^2-\rho_\ep^n(n-2)\omega_{n-1}A_\ep}{\be_\ep^2+\rho_\ep^2} \\
  &= \ell_\ep^2 \Big( \frac{1}{(n-2)\omega_{n-1}}\rho_\ep^{2-n} +A_\ep\Big)
  \frac{(\zeta_n^2 |A_\ep|-(n-2)\omega_{n-1}A_\ep)\rho_\ep^n}{\zeta_n^2|A_\ep|\rho_\ep^n+\rho_\ep^2} \\
  &= \ell_\ep^2 \Big( \frac{1}{(n-2)\omega_{n-1}} + A_\ep\rho_\ep^{n-2} \Big)
  \frac{\zeta_n^2 |A_\ep|-(n-2)\omega_{n-1}A_\ep}{\zeta_n^2|A_\ep|\rho_\ep^{n-2}+1}
\end{align*}
By assumption we have 
\begin{equation*}
  A_\ep \geq \xp{M}{a} +\ep.
\end{equation*}
If $\xp{M}{a}\leq 0$, then the assertion 1 of Theorem \ref{thm:xp_sigma} holds trivially. 
Thus we may assume that $A_\ep>0$ for all $\ep$.
We put 
\begin{equation*}
 \zeta_n := \frac{\sqrt{(n-2)\omega_{n-1}}}{2}.
\end{equation*}
If $\rho_\ep$ is small enough, we obtain 
\begin{align*}
  &\quad \int_{\partial B_\ep} u_\ep \frac{\partial u_\ep}{\partial r} da^{g_\ep} +  \ell_\ep^2 \Big(\frac{1}{(n-2)\omega_{n-1}}\rho_\ep^{2-n} + A_\ep\Big)\\
  &\leq \ell_\ep^2\frac{1}{(n-2)\omega_{n-1}}\Big(-\frac{(n-2)\omega_{n-1}A_\ep}{2}\Big).
\end{align*}
Inserting this into equation \eqref{eq:E_1+E_2} and using the definition \eqref{eq:gamma_definition} of $\gamma_\ep$ we get 
\begin{align*}
  E_1+E_2
  &\leq \sigma(S^n) \left( \int_{B_\ep} |\psi_\ep|^N dv^{g_\ep} \right)^{\frac{2}{N}}   \\
  & \quad{}-\ell_\ep^2\frac{1}{(n-2)\omega_{n-1}}\Big(-\frac{(n-2)\omega_{n-1}A_\ep}{2}+\frac{(n-2)\omega_{n-1}A_\ep}{4}\Big)\\
  &= \sigma(S^n) \left( \int_{B_\ep} |\psi_\ep|^N dv^{g_\ep} \right)^{\frac{2}{N}}
  -\frac{\ell_\ep^2A_\ep}{4}.
\end{align*}
Moreover, if $\rho_\ep$ is small enough, we obtain 
\begin{align*}
  \ell_\ep^2 
  &= \Big(\frac{\beta_\ep}{\be_\ep^2+\rho_\ep^2}\Big)^{n-2}
  \Big(\frac{1}{(n-2)\omega_{n-1}}\rho_\ep^{2-n}+A_\ep\Big)^{-2}\\
  &=\Big(\frac{\zeta_n |A_\ep|^{1/2}\rho_\ep^{n/2}}{\zeta_n^2A_\ep\rho_\ep^n+\rho_\ep^2}\Big)^{n-2}
  \Big(\frac{1}{(n-2)\omega_{n-1}}\rho_\ep^{2-n}+A_\ep\Big)^{-2}\\
  &=\Big(\frac{\zeta_n |A_\ep|^{1/2}\rho_\ep^{n/2}}{\zeta_n^2A_\ep\rho_\ep^{n-2}+1}\Big)^{n-2} \rho_\ep^{4-2n}
  \Big(\frac{1}{(n-2)\omega_{n-1}}+A_\ep\rho_\ep^{n-2}\Big)^{-2}\rho_\ep^{2n-4}\\
  &\geq \Big(\frac{\zeta_n |A_\ep|^{1/2}}{2}\Big)^{n-2}\rho_\ep^{\frac{n(n-2)}{2}} \frac{(n-2)^2\omega_{n-1}^2}{4}.
\end{align*}
From this and the definition \eqref{eq:rho_ep_definition} of $\rho_\ep$ it follows that 
\begin{align*}
  E_1+E_2 
  &\leq \sigma(S^n) \left( \int_{B_\ep} |\psi_\ep|^N dv^{g_\ep} \right)^{\frac{2}{N}}
  -4^{-n}\big((n-2)\omega_{n-1}\big)^{\frac{n+2}{2}}|A_\ep|^{\frac{n}{2}}\rho_\ep^{\frac{n(n-2)}{2}}\\
  &= \sigma(S^n) \left( \int_{B_\ep} |\psi_\ep|^N dv^{g_\ep} \right)^{\frac{2}{N}}
  -4^{-n}\big((n-2)\omega_{n-1}\big)^{\frac{n+2}{2}}
  \delta_n^{\frac{n}{2}}
  \frac{|A_\ep|^n}{(B_n+|A_\ep|)^n}\\
  &=:\sigma(S^n) \left( \int_{B_\ep} |\psi_\ep|^N dv^{g_\ep} \right)^{\frac{2}{N}}
  -E'_n \frac{|A_\ep|^n}{(B_n+|A_\ep|)^n}.
\end{align*}
We have 
\begin{equation*}
  D_\ep 
  \geq \left(\int_{B_\ep} |\psi_\ep|^N\, dv^{g_\ep}\right)^{\frac{2}{N}}
\end{equation*}
and therefore 
\begin{equation*}
  Q_\ep(\psi_\ep) 
  = \frac{E_1+E_2}{D_\ep}
  \leq \sigma(S^n) 
  -E'_n \frac{|A_\ep|^n}{(B_n+|A_\ep|)^n} \left(\int_{B_\ep} |\psi_\ep|^N\, dv^{g_\ep}\right)^{-\frac{2}{N}}
\end{equation*}
Moreover, using the substitution $r=\beta_\ep s$ 
we get 
\begin{align*}
  \int_{B_\ep} |\psi_\ep|^N\, dv^{g_\ep}
  &=\omega_{n-1}
  \int_0^{\rho_\ep} \Big(\frac{\beta_\ep}{\be_\ep^2+r^2}\Big)^{n} r^{n-1} \,dr\\
  &=\omega_{n-1}
  \int_0^{\rho_\ep/\be_\ep} \Big(\frac{1}{1+s^2}\Big)^{n} s^{n-1} \,ds\\
  & \leq \omega_{n-1}
  \int_0^{\infty} \Big(\frac{1}{1+s^2}\Big)^{n} s^{n-1} \,ds\\
  &=: F_n.
\end{align*}
It follows that
\begin{equation*}
  Q_\ep(\psi_\ep) 
  \leq \sigma(S^n) 
  -\frac{E'_n}{F_n^{2/N}} \frac{|A_\ep|^n}{(B_n+|A_\ep|)^n} 
\end{equation*} 
and therefore with $G_n:=(E'_nF_n^{-2/N})^{1/n}$ we get 
\begin{equation*}
  a-\ep 
  < Y(g_\ep)
  \leq \sigma(S^n) 
  -(G_n)^n \frac{|A_\ep|^n}{(B_n+|A_\ep|)^n}.
\end{equation*}
We take the limit $\ep \to 0$ and we get
\begin{equation*}
  a 
  \leq \sigma(S^n) 
  -(G_n)^n \frac{\xp{M}{a}^n}{(B_n+\xp{M}{a})^n}
  = \sigma(S^n) 
  -(G_n)^n \frac{1}{(\frac{B_n}{\xp{M}{a}}+1)^n}
\end{equation*}
It follows that
\begin{equation*}
  \frac{B_n}{\xp{M}{a}}+1
  \geq \frac{G_n}{(\sigma(S^n)-a)^{1/n}}
\end{equation*}
We put $\alpha_n:=\sigma(S^n)-\frac{(G_n)^n}{2^n}$ and we distinguish two cases.\\
a) If $a\geq\alpha_n$ we get 
\begin{equation*}
  \frac{B_n}{\xp{M}{a}}
  \geq \frac{1}{2} \frac{G_n}{(\sigma(S^n)-a)^{1/n}}
\end{equation*}
and since $a<\sigma(S^n)$ we get 
\begin{equation*}
  \xp{M}{a} 
  \leq \frac{2B_n}{G_n} (\sigma(S^n)-a)^{1/n}
  \leq \frac{2B_n}{G_n}\, \frac{\sigma(S^n)}{a}
  (\sigma(S^n)-a)^{1/n}
\end{equation*}
b) If $a\leq\alpha_n$ we have 
\begin{equation*}
  \sigma(S^n)-a \geq  \sigma(S^n)-\alpha_n
\end{equation*}
and therefore using equation \eqref{eq:xp_leq_C/a} 
\begin{equation*}
  \xp{M}{a} 
  \leq \frac{D_n^2}{4a}
  \leq \frac{D_n^2(\sigma(S^n)-a)^{1/n}}{4(\sigma(S^n)-\alpha_n)^{1/n}a}
  =\frac{D_n^2(\sigma(S^n)-a)^{1/n}}{2G_n a}.
\end{equation*}
This shows the second statement.
\end{proof}

\begin{appendix}

\section{A lemma on the Yamabe constant}

Let $M$ be a compact manifold of dimension $n\geq 3$.
For any Riemannian metric $g$ on $M$ we denote by
\begin{equation*}
  L_g:=\Delta_g+\frac{n-2}{4(n-1)}\scal_g 
\end{equation*}
the Yamabe operator of $g$ and by 
\begin{equation*}
Q_g:\quad C^{\infty}(M)\to C^{\infty}(M),\quad 
Q_g(u):=\frac{\int_M u L_g u\,dv^g}{(\int_M |u|^N)^{2/N}}
\end{equation*}
the Yamabe functional of $g$ where $N:=\frac{2n}{n-2}$.
We denote by 
\begin{equation*}
  Y(g):=\inf\{Q_g(u)\mid u\in C^{\infty}(M),\,u\neq 0\}
\end{equation*}
the Yamabe constant of $(M,g)$. 
In this Appendix, we prove the following Lemma which is used several times in the paper. 

\begin{lemma} \label{appendixlem}
  Let $(M,g)$ be a compact Riemannian manifold with $Y(g)>0$. 
  Let $p\in M$, let $\ep>0$ and let $h$ be a Riemannian metric defined on $B^{g}_p(2\ep)$. 
  We assume that $g$ and $h$ coincide at $p$. 
  Let $\chi_\ep \in C^{\infty}(M)$, $0 \leq \chi_\ep \leq 1$ be a cut-off function equal to $1$ on $B^{g}_p(\ep)$, equal to $0$ outside $B^{g}_p(2\ep)$ and which satisfies 
  $|d \chi_\ep|_g \leq \frac{C}{\ep}$ and $|\nabla^2 \chi_\ep| \leq \frac{C}{\ep^2}$. 
  Define 
  \begin{equation*}
    g_\ep:= \chi_\ep h + (1-\chi_\ep) g.
  \end{equation*}
  Then we have $\lim_{\ep\to 0} Y(g_{\ep})=Y(g)$.
\end{lemma}

\begin{proof}
  Note that by Taylor expansion we have on $B^{g}_p(2\ep)$
  \begin{equation*}
  | g-h|_g \leq C \ep \quad\text{and thus}\quad
  | g- g_\ep|_g \leq C \ep
  \end{equation*}
  for some $C$ independent of $\ep$.
  Thus for all $u\in C^{\infty}(M)$ we have 
  \begin{align*}
  \left|\int_{M} u\,dv^{g_\ep} - \int_{M} u\,dv^g \right| 
  &\leq C \ep \int_{B^g_p(2\ep)} |u|\, dv^g,\\
  \left| |du|^2_{g_\ep} - |du|^2_g \right| 
  &= \left| (g_\ep^{ij}-g^{ij}) \partial_i u \partial_ju \right| \leq C \ep |du|_g^2
  \end{align*}
  and using that $|\nabla^2\chi_\ep|\,|g-h|_g\leq\frac{C}{\ep}$ we get 
  \begin{equation*}
  |\scal_{g_\ep} - \scal_g| \leq \frac{C}{\ep}
  \quad\text{and}\quad
  |\scal_{g_\ep}|\leq\frac{C}{\ep}.
  \end{equation*}
  Let $(u_{\ep})_{\ep>0}$ be a sequence of smooth functions on $M$ such that $\int_M |u_\ep|^N\,dv^g$ and $Q_g(u_\ep)$ are bounded independently of $\ep$.
  We prove that 
  \begin{equation}
    \label{eq:lim_Q_g_ep_u_ep}
    \lim_{\ep\to 0} (Q_{g_\ep}(u_\ep)-Q_g(u_\ep))=0.
  \end{equation}
  By the triangle inequality we have
  \begin{align*}
    &\quad\left|\int_M \scal_{g_\ep}u_\ep^2\,dv^{g_\ep} - \int_M \scal_g u_\ep^2\,dv^g\right|\\
    &\leq \left|\int_M \scal_{g_\ep}u_\ep^2\,dv^{g_\ep} - \int_M \scal_{g_\ep} u_\ep^2\,dv^g\right|
    +\int_M |\scal_{g_\ep}-\scal_g| |u_\ep|^2\,dv^g
  \end{align*}
  By the H\"older inequality we get
  \begin{equation*}
    \int_M | \scal_{g_\ep} - \scal_g| |u_\ep|^2 dv^g \leq \frac{C}\ep \vol(B_p^{g}(2 \ep) )^{\frac{2}n}  \left( \int_M |u_\ep|^N dv^g \right)^{\frac{2}{N}} \leq C \ep  \left( \int_M |u_\ep|^N dv^g \right)^{\frac{2}{N}},
  \end{equation*}
  and
  \begin{align*} 
    &\quad \left|\int_M \scal_{g_\ep}u_\ep^2\,dv^{g_\ep}-\int_M \scal_{g_\ep} u_\ep^2\,dv^g\right|
    \leq C\ep \int_{B_p^{g}(2\ep)} |\scal_{g_\ep}| |u_\ep|^2\,dv^g\\
    &\leq C \vol(B_p^{g}(2 \ep) )^{\frac{2}n}  \left( \int_M |u_\ep|^N dv^g \right)^{\frac{2}{N}}
    \leq C \ep^2  \left( \int_M |u_\ep|^N dv^g \right)^{\frac{2}{N}}.
  \end{align*} 
  Moreover, we get
  \begin{align*}
    &\quad \left|\int_M |du_\ep|_{g_\ep}^2\,dv^{g_\ep}-\int_M |du_\ep|_g^2\,dv^g\right|\\
    &\leq \left|\int_M |du_\ep|_{g_\ep}^2\,dv^{g_\ep}-\int_M |du_\ep|_{g_\ep}^2\,dv^g\right|
      +\int_M \left|  |du_\ep|_{g_\ep}^2 - |du_\ep|_g^2\right|\,dv^g\\
    &\leq C\ep \int_{M} |du_\ep|_g^2\,dv^g\\
    &=C\ep \left(\int_{M} u_\ep L_g u_\ep\,dv^g 
      -\int_{M} \scal_g |u_\ep|^2\,dv^g\right)\\
    &\leq C\ep \left(\int_{M} u_\ep L_g u_\ep\,dv^g 
      + C\left(\int_{M} |u_\ep|^N\,dv^g\right)^{\frac{2}{N}}\right).
  \end{align*}
  The equation \eqref{eq:lim_Q_g_ep_u_ep} follows easily from these estimates. 
  Now let $(u_\ep)_{\ep}$ be a sequence in $C^{\infty}(M)$ such that $\int_{M}|u_\ep|^N\,dv^g=1$ and
  \begin{equation*}
    Q_{g_\ep}(u_\ep)=Y(g_\ep)+\ep
  \end{equation*}
  for all $\ep$. 
  Using \eqref{eq:lim_Q_g_ep_u_ep} we get 
  \begin{equation*}
    Y(g)
    \leq \liminf_{\ep\to 0} Q_g(u_\ep)
    = \liminf_{\ep\to 0} Q_{g_\ep}(u_\ep)
    = \liminf_{\ep\to 0}Y(g_\ep).
  \end{equation*}
  On the other hand let $u\in C^{\infty}(M)$ such that $\int_{M}|u|^N\,dv^g=1$ and  
  \begin{equation*}
    Q_g(u)
    = Y(g)+\delta
  \end{equation*}
  where $\delta>0$. 
  Using \eqref{eq:lim_Q_g_ep_u_ep} we get 
  \begin{equation*}
    \limsup_{\ep\to 0}Y(g_\ep)
    \leq \limsup_{\ep\to 0} Q_{g_\ep}(u)
    = \limsup_{\ep\to 0} Q_g(u)
    = Y(g)+\delta.
  \end{equation*}
  As $\delta\to 0$ we get $\limsup_{\ep\to 0}Y(g_\ep)\leq Y(g)$.
\end{proof}

\end{appendix}


\end{document}